\documentclass[12pt,reqno]{amsart}
\usepackage{amssymb,comment}
\usepackage{hyperref}
\usepackage{stmaryrd}
\usepackage{slashed}
\usepackage{url}

\newcounter{count}[section]
\renewcommand{\thecount}{\arabic{section}.\arabic{count}}

\numberwithin{equation}{section}
    
\newenvironment{Umgeb1rekursiv}[1]
   {\vspace{0.5cm}
     \noindent
     \par\noindent
     \refstepcounter{count}
     \textbf{#1~\thecount}
     \hspace{0.2cm}
     \itshape
   }
   {\vspace{0.2cm}\par}

\newenvironment{Umgeb1}[1]
   {\vspace{0.5cm}
     \par\noindent
     \refstepcounter{count}
     \textbf{#1~\thecount}
     \hspace{0.2cm}
   }
   {\vspace{0.2cm}\par}

\newenvironment{prop}{\begin{Umgeb1rekursiv}{Proposition}}
  {\end{Umgeb1rekursiv}}
\newenvironment{lem}{\begin{Umgeb1rekursiv}{Lemma}}
  {\end{Umgeb1rekursiv}}
\newenvironment{corollary}{\begin{Umgeb1rekursiv}{Corollary}}
  {\end{Umgeb1rekursiv}}  
\newenvironment{theorem}{\begin{Umgeb1rekursiv}{Theorem}}
  {\end{Umgeb1rekursiv}}

\newenvironment{remark}{\begin{Umgeb1}{Remark}}
   {\end{Umgeb1}}
   
\newcommand{\R}{\mathbb{R}}

\newcommand{\C}{\mathbb{C}}
\newcommand{\N}{\mathbb{N}}

\newcommand{\id}{\operatorname{Id}}

\newcommand{\dm}{d}

\newcommand{\Res}{\operatorname{Res}}

\newcommand{\abs}[1]{\lvert#1\rvert}
\renewcommand{\part}[2]{\frac{\partial #1}{\partial #2}}

\def\st{\stackrel{\text{def}}{=}}

\title{A family of Riesz distributions for differential forms on Euclidian space}

\begin{document}
\author{Matthias Fischmann and Bent {\O}rsted}

\address{Departmant of Mathematics, Ny Munkegade 118, 8000 Aarhus C, Denmark}

\thanks{Research supported by the Department of Mathematics (Aarhus University) and the Danish Research Council}

\email{fischmann@math.au.dk, orsted@math.au.dk}



\begin{abstract}
  In this paper we introduce a new family of operator-valued distributions
  on Euclidian space acting by convolution on differential forms. It provides a natural
  generalization of the important Riesz distributions acting on functions,
  where the corresponding operators are $(-\Delta)^{-\alpha/2}$, and we develop 
  basic analogous properties with respect to meromorphic continuation, residues,
  Fourier transforms, and relations to conformal geometry and representations
  of the conformal group.
\end{abstract}
\maketitle

\tableofcontents
\allowdisplaybreaks


\section{Introduction}
  Important aspects of analysis in Euclidian space are studied via the Riesz distributions
  $|x|^{\lambda}$, see \cite{Riesz}, i.e. the complex powers of the Euclidian norm. As convolution operators 
  these allow a rigorous treatment of the Laplace operator $\Delta$ and its complex powers
  $I_{\alpha} = (-\Delta)^{-\alpha/2}$, acting on functions
as a natural semigroup, and there are many classical and recent results related
  to this family of operators, sometimes referred to as fractional Laplacians. 
  
  Generalizations of those ideas to norms induced by indefinite metrics or spinor valued distributions 
  where studied in \cite{KV} and \cite{CO}.
  
  In this paper we introduce a natural family of similar operators acting on differential forms
  in Euclidian space; however, the semigroup property (as above) has to be relaxed.
  The corresponding family of distributions seems to deserve the name of
  Riesz distributions for differential forms, and it allows a certain extra flexibility in the complex
  parameters involved. We shall start by giving the basic definitions and calculations in the
  most general cases, and later specialize to some specific cases of particular interest.
  One of the important special cases is not new by any means, since it corresponds to
  some of the cases of the intertwining operators treated by Knapp and Stein in detail, when they
  introduced their celebrated kernel operators \cite{KnappStein}. Here our 
  calculations make these operators \ref{eq:Riesz2} more explicit; in particular we give the Euclidian Fourier
  transform of these Knapp-Stein operators and simplifying proofs for the unitarity of the
  complementary series of representations of the conformal group. For this 
  see Remark \ref{ComplementarySeries}. Another interesting case which arises by specializing 
  parameters in our family is the Beurling-Ahlfors operator $S$ 
  (in even dimension, acting on forms of middle degree), see \cite{IM}; this operator 
  is in several ways an analogue of the Hilbert transform.

  In summary, we shall for the Riesz distributions for differential forms \eqref{eq:GeneralRiesz} be interested in their
  \begin{itemize}
    \item Fourier transforms (Theorem \ref{FourierGeneralRiesz})
    \item Bernstein-Sato identities (Theorem \ref{BernsteinSatoGeneralRiesz})
    \item residues (Theorem \ref{ResiduesGeneralRiesz})
    \item convolution formulas (Theorem \ref{ConvolutionIndentityGeneralRiesz})
    
  \end{itemize}
  and for example obtain the Branson-Gover operators \cite{BransonGover}
  \begin{align*}
    L^{(p)}_{2N} = (\frac{n}{2} - p + N)(\delta d)^N + (\frac{n}{2} - p - N)(d \delta)^N
  \end{align*}
  on $p$-forms in Euclidian space (well-known from conformal geometry) as residues, see Corollary \ref{ResiduesRiesz2}. 
  Here $d$ is the usual
  derivative on forms and $\delta$ its $L^2$-adjoint, and these (and their symbols) are
  the basic building blocks in our study. 
  This corresponds well with the conjectural fact that in general
  (on general Riemannian manifolds) the Branson-Gover operators may be obtained as
  residues of the scattering operator \cite{AG, GZ}. 
  While we shall carry out most of the
  analysis in Euclidian space, it is also possible to work on the conformal compactification,
  i.e. the standard sphere of the same dimension; for this we shall apply some of the
  formulas in \cite{BOO}, where the so-called compact picture of the bundles in question
  are studied.

  We see the present paper as laying the groundwork for further studies of Euclidian analysis
  on differential forms, such as for example Sobolev inequalities, 
  giving the basic facts towards finding fundamental solutions
  of natural differential operators on differential forms, and elliptic boundary value problems.
  Also we hope our study might contribute to the branching program of T. Kobayashi, in
  particular to his theory of symmetry-breaking operators and branching problems for
  complementary series representations.

\section{Preliminaries about differential forms}
  Let $(\R^n,\langle\cdot,\cdot\rangle)$ be the Euclidian vector space of dimension $n\in\N$ and 
  $\{e_j\}_{j=1}^n$ its standard basis. 
  The space of differential $p$-forms on $\R^n$, for $0\leq p\leq n$, is denoted by 
  $\Omega^p(\R^n)$. The Euclidian scalar product extends to $\Omega^p(\R^n)$ and will be denoted with 
  the same symbol. We introduce two 
  algebraic actions on $\omega\in\Omega^p(\R^n)$ which will be important: For $x\in\R^n$ one defines 
  \begin{align*}
    i_x\omega&\st\sum_{k=1}^nx_k i_{e_k}\omega,\quad
    \varepsilon_x\omega\st\sum_{k=1}^n x_k e_k\wedge \omega.
  \end{align*}
  Mainly these are the interior and exterior product by $x\in\R^n$ and are 
  related as symbols to
  the exterior differential $\dm:\Omega^p(\R^n)\to\Omega^{p+1}(\R^n)$ 
  and the negative of the co-differential $\delta:\Omega^p(\R^n)\to\Omega^{p-1}(\R^n)$, the $L^2$-adjoint of the differential.
  They satisfy the well known identities:
  \begin{lem}\label{Identities1}
    The algebraic actions $i_x$ and $\varepsilon_x$ are nilpotent of degree $2$ and
    formally adjoint to each other with respect to the scalar product 
    $\langle\cdot,\cdot\rangle$ on $\Omega^p(\R^n)$, 
    i.e. $\langle i_x\omega,\eta\rangle=\langle\omega ,\varepsilon_x \eta \rangle$ 
    for all $x\in\R^n$ and $\omega\in\Omega^p(\R^n)$, $\eta\in\Omega^{p-1}(\R^n)$. 
    Furthermore, it holds
    \begin{align*}
      i_x\varepsilon_x&=\sum_{k,l=1}^nx_k x_l i_{e_k}(e_l\wedge\cdot),\quad
      \varepsilon_xi_x=\sum_{k,l=1}^nx_k x_l e_l\wedge i_{e_k}(\cdot),\quad
      i_x\varepsilon_x+\varepsilon_xi_x=\sum_{k=1}^n x_k^2.
    \end{align*}
  \end{lem}
  In the next section it will be important to know some differential actions on 
  the distribution 
  \begin{align}\label{eq:RieszDistribution}
    r^\lambda(x)\st (x_1^2+\cdots+x_n^2)^{\frac \lambda 2}
  \end{align}
  defined for $x\in\R^n$ and $\lambda\in\C$ with $\Re(\lambda)>-n$. This distribution is termed 
  {\it Riesz distribution} and was for example studied in \cite{Riesz,GelfandShilov}. 

  The following is the key lemma in order to obtain the Fourier transform of the Riesz distribution for differential forms. 
  \begin{lem}\label{Identities2}
    Let $\beta$ be a constant $p$-form. Then, for fixed $x\in\R^n$, 
    \begin{align*}
      \delta\dm(r^{\lambda+2}(x-y)\beta)
        &= -(\lambda+2)(n-p)r^{\lambda}(x-y)\beta
        -(\lambda+2)\lambda r^{\lambda-2}(x-y)i_{x-y}\varepsilon_{x-y}\beta ,\\
      \dm\delta(r^{\lambda+2}(x-y)\beta)
        &= -(\lambda+2)pr^{\lambda}(x-y)\beta
        -(\lambda+2)\lambda r^{\lambda-2}(x-y)\varepsilon_{x-y}i_{x-y}\beta.
    \end{align*}
    Derivatives are taken with respect to the $y$-variable. 
  \end{lem}
  \begin{proof}
    In order to prove the lemma we first observe 
    \begin{align}\label{eq:Observation}
      \partial_k r^\nu(x-y)&=-\nu (x_k-y_k)r^{\nu-2}(x-y),\notag\\
        \sum_{k=1}^n e_k\wedge i_{e_k}\beta&=p\beta,\quad
        \sum_{k=1}^n i_{e_k}(e_k\wedge \beta)=(n-p)\beta.
    \end{align}
    Then a straightforward computation shows 
    \begin{align*}
      \delta\dm(r^{\lambda+2}(x-y)\beta)
        &=-\sum_{k,l=1}^n \partial_l\partial_k (r^{\lambda+2}(x-y))
        i_{e_k}(e_l\wedge \beta) \\
      &=-(\lambda+2)r^\lambda(x-y) \sum_{k=1}^n i_{e_k}(e_k\wedge \beta) \\
      &-\lambda(\lambda+2)r^{\lambda-2}(x-y) \sum_{k,l=1}^n (x_k-y_k)(x_l-y_l)
        i_{e_k}(e_l\wedge \beta) 
    \end{align*}
    Using the above observations \eqref{eq:Observation} and Lemma \ref{Identities1} we may conclude 
    \begin{align*}
      \delta\dm(r^{\lambda+2}(x-y)\beta)= -(\lambda+2)(n-p)r^{\lambda}(x-y)
        \beta-(\lambda+2)\lambda r^{\lambda-2}(x-y)i_{x-y}\varepsilon_{x-y}\beta.
    \end{align*}
    The second claim runs along the same line, which completes the proof.
  \end{proof}
  
  Let us denote by $\mathcal{S}^p(\R^n)\st \mathcal{S}(\R^n)\otimes \Lambda^p(\R^n)$ the 
  space of Schwartz functions with value in $\Lambda^p(\R^n)^*$. 
  We follow the convention for the Fourier 
  transform \cite{GelfandShilov} on Schwartz functions $f\in\mathcal{S}(\R^n)$:
  \begin{align*}
    \mathcal{F}(f)(\xi)\st \int_{\R^n} f(x)e^{i \langle x, \xi\rangle}dx,
  \end{align*}
  and extend it to a Fourier transform on 
  $\omega\in\mathcal{S}^p(\R^n)$ by acting on the coefficients of $\omega$ in some arbitrarily chosen basis. Our 
  normalization of the Fourier transform is chosen in such a way that 
  $\mathcal{F}(\delta_0)=1$, where $\delta_0$ is the Dirac-distribution centered at the origin. 
  Recall that for a polynomial $P$ in $n$ variables we have the identities
  \begin{align}
    P(\partial_{\xi_1},\ldots,\partial_{\xi_n})\mathcal{F}(\omega)(\xi)
      &=\mathcal{F}(P(i x_1,\ldots,i x_n)\omega)(\xi),\notag\\
    \mathcal{F}(P(\partial_{x_1},\ldots,\partial_{x_n})\omega)(\xi)
      &=P(-i \xi_1,\ldots,-i \xi_n)\mathcal{F}(\omega)(\xi)\label{eq:FourierProperties2}
  \end{align}
  for $\omega\in\mathcal{S}^p(\R^n)$.

\section{A class of Riesz distributions for differential forms}
  We introduce a class of tempered distributions 
  $R_{A_\lambda,B_\lambda}^\lambda(x)$ valued in 
  endomorphisms of differential forms, which generalizes the 
  Riesz distribution \eqref{eq:RieszDistribution}. 
  Based on an explicit formula for its Fourier transform, we give a
  corresponding Bernstein-Sato identity, compute its residues and some identities concerning their convolutions. Finally, 
  this distribution when acting by convolution on differential forms induces an integral operator.  

  Let $A_\lambda,B_\lambda$ be holomorphic functions in $\lambda\in\C$. 
  For $\lambda\in \C$ with $\Re(\lambda)>-n$ a distribution valued in the endomorphisms of differential forms, 
  termed {\it Riesz distribution for differential forms}, is defined by
  \begin{align}\label{eq:GeneralRiesz}
    R_{A_\lambda,B_\lambda}^\lambda(x)\st r^{\lambda-2}(x)
     (A_\lambda i_x \varepsilon_x  
      + B_\lambda\varepsilon_x i_x).
  \end{align}
  The parameters $A_\lambda,B_\lambda$ can also depend on the form-degree 
  $p$ and the dimension $n$, but we will suppress such 
  dependencies. 
  \begin{remark}
    One could also assume that $A_\lambda$ and $B_\lambda$ are meromorphic in $\lambda$ without 
    disturbing the results obtained here. Of course, possible poles of $A_\lambda$ and $B_\lambda$ 
    will influence the assumtions and statements in this paper.
  \end{remark}

  
 \subsection{Fourier transform} 
  It is well known that the Fourier transform, a special version of an integral transform, plays an 
  important rule in the analysis of functions. Into that branch falls for example the study of solutions of 
  differential equations, detecting possible poles and computing their residues. 
  
  In terms of $A_\lambda,B_\lambda$ we define holomorphic functions 
  \begin{align}\label{eq:Coefficients1}
    C_\lambda\st (\lambda+p)A_\lambda-pB_\lambda,
      \quad D_\lambda\st -(n-p)A_\lambda
      +(\lambda+n-p)B_\lambda.
  \end{align}
  
  The next theorem states that the Fourier transform $\mathcal{F}$ preserves the class 
  of distribution $R^\lambda_{A_\lambda,B_\lambda}(x)$ 
  by changing (in general) the parameters $A_\lambda,B_\lambda$. 
  \begin{theorem}\label{FourierGeneralRiesz}
    The Fourier transform of $R_{A_\lambda,B_\lambda}^\lambda(x)$ is given by 
    \begin{align}\label{eq:FourierGeneralRiesz}
      \mathcal{F}(R_{A_\lambda,B_\lambda}^{\lambda})(\xi)
        = -2^{\lambda+n-1}\pi^{\frac n2}
         \frac{\Gamma(\frac{\lambda+n}{2})}{\Gamma(-\frac{\lambda-2}{2})} 
         R_{C_\lambda,D_\lambda}^{-\lambda-n}(\xi),
    \end{align}
    where $C_\lambda,D_\lambda$ are given by \eqref{eq:Coefficients1}.
  \end{theorem}
  \begin{proof}
    For $\omega\in\mathcal{S}^p(\R^n)$ it is enough to verify that both expressions
    \begin{align*}
      \mathcal{F}\Big(\int_{\R^n}r^{\lambda-2}(x-y)
       (A_\lambda i_{x-y} \varepsilon_{x-y} + B_\lambda\varepsilon_{x-y} i_{x-y})
       \omega(y)\dm y\Big)(\xi),
    \end{align*}
    which is actually $\mathcal{F}(R_{A_\lambda,B_\lambda}^{\lambda} * \omega)(\xi)$, and 
    \begin{align*}
      r^{-\lambda-2-n}(\xi)(\tilde{C}_\lambda i_\xi\varepsilon_\xi 
         +\tilde{D}_\lambda \varepsilon_\xi i_\xi)\mathcal{F}(\omega)(\xi),
    \end{align*}
    for $\tilde{C}_\lambda,\tilde{D}_\lambda\in\C$ to be determined, 
    agree up to a constant. We start to compute the latter one. First note the identities:
    \begin{align*}
      r^{-\lambda-2-n}(\xi)&=c_\lambda^{-1}\mathcal{F}(r^{\lambda+2})(\xi),\\
      (\tilde{C}_\lambda i_\xi\varepsilon_\xi 
         +\tilde{D}_\lambda \varepsilon_\xi i_\xi)\mathcal{F}(\omega)(\xi)
         &=\mathcal{F}\big((\tilde{C}_\lambda\delta\dm
         +\tilde{D}_\lambda\dm\delta)\omega\big)(\xi),
    \end{align*}
    see \cite[Chapter II, Section $3.3$]{GelfandShilov} and \eqref{eq:FourierProperties2}, and 
    \begin{align}\label{eq:FourierRieszConstant}
      c_\lambda\st 2^{\lambda+2+n}\pi^{\frac n2}\frac{\Gamma(\frac{\lambda+2+n}{2})}{\Gamma(-\frac{\lambda+2}{2})},
    \end{align} 
    Hence we obtain 
    \begin{align*}
      r^{-\lambda-2-n}(\xi)(\tilde{C}_\lambda i_\xi\varepsilon_\xi 
         +\tilde{D}_\lambda \varepsilon_\xi i_\xi)\mathcal{F}(\omega)(\xi)
        = c_\lambda^{-1}\mathcal{F}\big(\int_{\R^n} 
        r^{\lambda+2}(x-y)(\tilde{C}_\lambda\delta\dm
        +\tilde{D}_\lambda\dm\delta)\omega \dm y \big)(\xi).
    \end{align*}
    Taking a constant $p$-form $\beta$ we have by partial integration 
    \begin{multline}\label{eq:help1}
      \int_{\R^n}\langle \beta,r^{\lambda+2}(x-y)(\tilde{C}_\lambda\delta\dm
        +\tilde{D}_\lambda\dm\delta)\omega \rangle \dm y\\
      = \int_{\R^n}\langle(\tilde{C}_\lambda\delta\dm
        +\tilde{D}_\lambda\dm\delta)( r^{\lambda+2}(x-y)\beta),\omega \rangle \dm y.
    \end{multline}
    By Lemma \ref{Identities2} we find 
    \begin{align*}
      \int_{\R^n}\langle \delta\dm(r^{\lambda+2}(x-y)\beta),\omega\rangle \dm y
      =-(\lambda+2)\int_{\R^n}\langle\beta&, (n-p) r^{\lambda}(x-y)\omega\\
      &+\lambda r^{\lambda-2}(x-y)
        i_{x-y}\varepsilon_{x-y}\omega\rangle\dm y,\\
      \int_{\R^n}\langle \dm\delta(r^{\lambda+2}(x-y)\beta),\omega\rangle \dm y
      =-(\lambda+2)\int_{\R^n}\langle\beta&, pr^{\lambda}(x-y)\omega\\
      &+\lambda r^{\lambda-2}(x-y)
       \varepsilon_{x-y}i_{x-y}\omega\rangle\dm y. 
    \end{align*}
    In turn, Equation \eqref{eq:FourierGeneralRiesz} is equivalent by 
    collecting  the coefficients of $r^{\lambda-2}(x-y)i_{x-y}\varepsilon_{x-y}$ 
    and $r^{\lambda-2}(x-y)\varepsilon_{x-y}i_{x-y}$ in \eqref{eq:help1} to the following 
    linear system for $\tilde{C}_\lambda, \tilde{D}_\lambda$:
    \begin{align}\label{eq:CoeffRelationsForFourier}
      (\lambda+n-p)\tilde{C}_\lambda+p \tilde{D}_\lambda=A_\lambda,\notag\\ 
      (n-p)\tilde{C}_\lambda+(\lambda+p)\tilde{D}_\lambda=B_\lambda.
    \end{align}
    The unique solution is given by 
    \begin{align}\label{eq:SolutionCoeffRelationsForFourier}
      \tilde{C}_\lambda=\frac{(\lambda+p)A_\lambda-pB_\lambda}{\lambda(\lambda+n)},
        \quad \tilde{D}_\lambda=\frac{-(n-p)A_\lambda+(\lambda+n-p)B_\lambda}{\lambda(\lambda+n)}
    \end{align} 
    This implies, since $\beta$ was an arbitrary $p$-form, 
    \begin{multline*}
      c_\lambda^{-1}\mathcal{F}\big(\int_{\R^n} r^{\lambda+2}(x-y)
        ( \tilde{C}_\lambda \delta\dm+\tilde{D}_\lambda\dm\delta)\omega \dm y \big)(\xi)\\
        =-(\lambda+2)c_\lambda^{-1}
           \int_{\R^n}r^{\lambda-2}(x-y)(A_\lambda i_{x-y}\varepsilon_{x-y}
           +B_\lambda \varepsilon_{x-y} i_{x-y})\omega(y)\dm y.
    \end{multline*}
    Note that the factor $\lambda(\lambda+n)$ in $\tilde{C},\tilde{D}$ will be absorbed into 
    corresponding Gamma functions arising from $c_\lambda$.
    The proof is complete. 
  \end{proof}

 \subsection{Bernstein-Sato identity}
  For a given polynomial $f\in\R[x_1,\ldots,x_n]$ and complex number $s\in\C$
  one can consider $f_+^s(x):=f(x)^s$ for $f(x)> 0$ and 
  zero otherwise. If $\Re(s)>0$ this is locally integrable on $\R^n$. Bernstein 
  \cite{Bernstein} proves that $f_+^s$ admits a meromorphic continuation to $\C$ with poles 
  given by the zero locus of a certain polynomial, the {\it Bernstein polynomial}. 
  Roughly speaking one can construct a differential operator $P(s)$ with polynomial coefficients such that 
  $P(s)f^{s+1}_+(x)=b(s) f^s_+(x)$, where $b(s)$ is the Bernstein-Sato polynomial. This differential equation 
  is termed {\it Bernstein-Sato identity} and was independently discovered in \cite{SatoShintani}. 
  
  Now we present a vector-valued Bernstein-Sato identity  
  which in turn implies the mereomorphic continuation of $R^\lambda_{A_\lambda,B_\lambda}(x)$ to $\lambda\in\C$. 
  \begin{theorem}\label{BernsteinSatoGeneralRiesz}
    The distribution $R^\lambda_{A_\lambda,B_\lambda}(x)$ 
    satisfies the differential equation:
    \begin{align}\label{eq:BersteinSatoGeneralRiesz}
      D_2(\lambda;n,p)R^{\lambda+2}_{A_{\lambda+2},B_{\lambda+2}}(x)
        =-\lambda(\lambda+n) C_{\lambda+2} D_{\lambda+2}
        R^{\lambda}_{A_{\lambda},B_{\lambda}}(x)
    \end{align}
    where $D_2(\lambda;n,p)$ is a second-order differential operator given by
    \begin{align*}
      D_2(\lambda;n,p)\st C_{\lambda}D_{\lambda+2} \delta\dm
       + D_{\lambda}C_{\lambda+2}\dm\delta.
    \end{align*}
  \end{theorem}
  
  \begin{proof}
    We apply the Fourier transform to Equation \eqref{eq:BersteinSatoGeneralRiesz} and use  
    Theorem \ref{FourierGeneralRiesz}.
    Hence, it remains to verify that   
    \begin{multline*}
       -2^{\lambda+n+1}\pi^{\frac n2}\frac{\Gamma(\frac{\lambda+n+2}{2})}{\Gamma(-\frac{\lambda}{2})}
         r^{-\lambda-n-4}(\xi)(C_{\lambda}D_{\lambda+2} i_\xi\varepsilon_\xi 
         +D_{\lambda}C_{\lambda+2} \varepsilon_\xi i_\xi)
         (C_{\lambda+2} i_\xi\varepsilon_\xi +D_{\lambda+2} \varepsilon_\xi i_\xi)\\
       =-2^{\lambda+n+1}\pi^{\frac n2}\frac{\Gamma(\frac{\lambda+2+n}{2})}{\Gamma(-\frac{\lambda}{2})}
              r^{-\lambda-n-4}(\xi)(C_\lambda C_{\lambda+2}D_{\lambda+2} (i_\xi\varepsilon_\xi)^2 
         +C_{\lambda+2} D_\lambda D_{\lambda+2}(\varepsilon_\xi i_\xi)^2)
    \end{multline*}
    and 
    \begin{multline*}
      \lambda(\lambda+n)C_{\lambda+2}D_{\lambda+2}2^{\lambda+n-1}\pi^{\frac n2}
         \frac{\Gamma(\frac{\lambda+n}{2})}{\Gamma(-\frac{\lambda-2}{2})}
         r^{-\lambda-n-2}(\xi)(C_{\lambda} i_\xi\varepsilon_\xi +D_{\lambda} \varepsilon_\xi i_\xi)\\
      =-2^{\lambda+n+1}\pi^{\frac n2}\frac{\Gamma(\frac{\lambda+n+2}{2})}{\Gamma(-\frac{\lambda}{2})}
              r^{-\lambda-n-4}(\xi)(C_\lambda C_{\lambda+2}D_{\lambda+2} (i_\xi\varepsilon_\xi)^2 
        +C_{\lambda+2}D_\lambda D_{\lambda+2}(\varepsilon_\xi i_\xi)^2)
    \end{multline*}
    do agree. 
  \end{proof}
  
  An iteratively application of Equation \eqref{eq:BersteinSatoGeneralRiesz} gives 
  \begin{align}\label{eq:IterationOfBernsteinSatoGeneralRiesz}
    R^{\lambda}_{A_\lambda,B_\lambda}(x)
       =\frac{(-1)^k}{4^k(\frac{\lambda}{2})_k(\frac{\lambda+n}{2})_k C_{\lambda+2k}D_{\lambda+2k}}
       D_{2k}(\lambda;n,p)R^{\lambda+2k}_{A_{\lambda+2k},B_{\lambda+2k}}(x),
  \end{align}
  where $D_{2k}(\lambda;n,p)$ is a differential operator of order $2k$, for $k\in \N$, given by 
  \begin{align}\label{eq:2kthOrderDifferential}
    D_{2k}(\lambda;n,p)\st C_{\lambda}D_{\lambda+2k}(\delta\dm)^k
     + C_{\lambda+2k}D_{\lambda}(\dm\delta)^k.
  \end{align}
  By convention we set $D_0(\lambda;n,p)\st C_\lambda D_\lambda\id$. 

  The tempered distribution $R^\lambda_{A_\lambda,B_\lambda}(x)$ originally defined for 
  $\Re(\lambda)>-n$ can now, by \eqref{eq:IterationOfBernsteinSatoGeneralRiesz}, 
  be meromorphically extended to $\lambda\in\C$ with simple poles at $\lambda=-n-2\N_0$.

 \subsection{Residues}
  In the last subsection we have identified the location of the poles of 
  $R^\lambda_{A_\lambda,B_\lambda}$ and their type. In general their residues, when acting as convolution operators, 
  will not be differential operators. However, under some assumtions they reduce to differential operators given 
  as a weighted sum of powers of $\delta\dm$ and $\dm\delta$.   
  \begin{theorem}\label{ResiduesGeneralRiesz}
    The residue of $\mathcal{F}(R^\lambda_{A_\lambda,B_\lambda})(\xi)$ at $\lambda=-n-2k$, for $k\in\N_0$ is 
    given by 
    \begin{multline}\label{eq:ResiduesGeneralRiesz}
      \Res_{\lambda=-n-2k}(\mathcal{F}(R^\lambda_{A_\lambda,B_\lambda})(\xi))\\
        =\frac{(-1)^{k+1}\pi^{\frac n2}}{4^{k} k! \Gamma(\frac{n+2k+2}{2})}
           \Big[\frac{C_{-n-2k}}{C_{-n}}(i_\xi\varepsilon_\xi)^k
           + \frac{D_{-n-2k}}{D_{-n}}(\varepsilon_\xi i_\xi)^k\Big]R^{0}_{C_{-n}, D_{-n}}(\xi).
    \end{multline}
    Especially, in case $R^\lambda_{A_\lambda,B_\lambda}(x)$ acts on $0$-forms we have 
    \begin{align}\label{eq:ResiduesGeneralRieszOnZeroForms}
     \Res_{\lambda=-n-2k}\big(\mathcal{F}(R^\lambda_{A_\lambda,B_\lambda})(\xi)\big)
        =\frac{(-1)^{k+1}\pi^{\frac n2}}{4^{k} k! \Gamma(\frac{n+2k+2}{2})}C_{-n-2k}r^{2k}(\xi).
    \end{align}
  \end{theorem}
  \begin{proof}
    First note that the residues of $\mathcal{F}(R^\lambda_{A_\lambda,B_\lambda})(\xi)$ 
    coincide with the residues of $R^\lambda_{A_\lambda,B_\lambda}(x)$.
    The poles at $\lambda=-n-2k$ are simple. Hence its residue is given by 
    \begin{align*}
      \Res_{\lambda=-n-2k}\big(\mathcal{F}(R^\lambda_{A_\lambda,B_\lambda})(\xi)\big)
        =\lim_{\lambda\to -n-2k}(\lambda+n+2k)\mathcal{F}(R^\lambda_{A_\lambda,B_\lambda})(\xi).
    \end{align*}
   From Equation \eqref{eq:IterationOfBernsteinSatoGeneralRiesz} and Theorem 
   \ref{FourierGeneralRiesz} we get 
    \begin{align*}
      \mathcal{F}(R^\lambda_{A_\lambda,B_\lambda})(\xi)
        &=\frac{(-1)^k\mathcal{F}(D_{2k}(\lambda;n,p)R^{\lambda+2k}_{A_{\lambda+2k},B_{\lambda+2k}})(\xi)}{4^k(\frac{\lambda}{2})_k(\frac{\lambda+n}{2})_kC_{\lambda+2k}D_{\lambda+2k}}\\
        &=\frac{(-1)^{k+1}2^{\lambda+2k+n-1}\pi^{\frac n2}\Gamma(\frac{\lambda+n+2k}{2})}{4^k (\frac{\lambda}{2})_k(\frac{\lambda+n}{2})_kC_{\lambda+2k}D_{\lambda+2k}\Gamma(-\frac{\lambda+2k-2}{2})}\\
      &\times \Big[C_{\lambda}D_{\lambda+2k}(i_\xi\varepsilon_\xi)^k
        + C_{\lambda+2k}D_{\lambda}(\varepsilon_\xi i_\xi)^k\Big]R^{-\lambda-2k-n}_{C_{\lambda+2k}, D_{\lambda+2k}}(\xi).
    \end{align*}
    Combining the factor $(\lambda+n+2k)$ with the Gamma function $\Gamma(\frac{\lambda+n+2k}{2})$ 
    and taking the limit $\lambda\to -n-2k$ gives 
    \begin{multline*}
      \Res_{\lambda=-n-2k}\big(\mathcal{F}(R^\lambda_{A_\lambda,B_\lambda})(\xi)\big)\\
        =\frac{(-1)^{k+1}\pi^{\frac n2}}{4^k k! C_{-n}D_{-n}\Gamma(\frac{n+2k+2}{2})}\Big[C_{-n-2k}D_{-n}(i_\xi\varepsilon_\xi)^k
          + C_{-n}D_{-n-2k}(\varepsilon_\xi i_\xi)^k\Big]R^{0}_{C_{-n}, D_{-n}}(\xi),
    \end{multline*} 
    which implies \eqref{eq:ResiduesGeneralRiesz}. The case when $\mathcal{F}(R^\lambda_{A_\lambda,B_\lambda})(\xi)$ acts on $0$-forms is an easy consequence. 
    The proof is complete.  
  \end{proof}
  
  \begin{corollary}\label{ResiduesGeneralRieszWithAssumption}
    Let $A_\lambda$ and $B_\lambda$ holomorphic in $\lambda$ such that $C_{-n}=D_{-n}$. Then the 
    residue of $R^\lambda_{A_\lambda,B_\lambda}(x)$ at $\lambda=-n-2k$, for $k\in\N_0$, is given by 
    \begin{align}
      \Res_{\lambda=-n-2k}\big(R^\lambda_{A_\lambda,B_\lambda}(x)\big)
        =\frac{(-1)^{k+1}\pi^{\frac n2}}{4^{k} k! \Gamma(\frac{n+2k+2}{2})}\Big[C_{-n-2k}(\delta\dm)^k
          + D_{-n-2k}(\dm\delta)^k\Big]\delta_0(x).
    \end{align}
  \end{corollary}
  \begin{proof}
    This is a direct consequence of Theorem \ref{ResiduesGeneralRiesz} and  
    \begin{align*}
      R^{0}_{C_{-n}, D_{-n}}(\xi)=C_{-n}r^{-2}(\xi)(i_\xi\varepsilon_\xi+\varepsilon_\xi i_\xi)=C_{-n}\id
    \end{align*}
    and $\mathcal{F}(\delta_0)(\xi)=\id(\xi)$. 
  \end{proof}

  The assumption in Corollary \ref{ResiduesGeneralRieszWithAssumption} is not empty. 
  For example, the pairs $(A_\lambda,B_\lambda)\st (1,1), (1,-1)$ or 
  $(\alpha_\lambda,\beta_\lambda)$, where $\alpha_\lambda,\beta_\lambda$ are given by 
  \eqref{eq:BGCoefficients}, satisfy the condition $C_{-n}=D_{-n}$, cf. \eqref{eq:Coefficients1}.

 \subsection{Convolutions}
  Already mentioned before, the operation between function known as convolution is 
  important to define certain integral operators. It also is important in the theory of Fourier transform, since 
  the Fourier transform transforms convolutions into products. 
  
  We proceed with an identity concerning convolutions of $R^\lambda_{A_\lambda,B_\lambda}(x)$.
  \begin{theorem}\label{ConvolutionIndentityGeneralRiesz}
    Let $A_\lambda,B_\lambda$ such that $C_{2(\lambda-n)}C_{-2\lambda}=D_{2(\lambda-n)}D_{-2\lambda}$. 
    Then the distribution $R^\lambda_{A_\lambda,B_\lambda}(x)$ satisfies
    \begin{align}\label{eq:ConvolutionIndentityGeneralRiesz}
      R^{2(\lambda-n)}_{A_{2(\lambda-n)},B_{2(\lambda-n)}}*R^{-2\lambda}_{A_{-2\lambda},B_{-2\lambda}}(x)
      =\pi^n \frac{\Gamma(\frac{2\lambda-n}{2})\Gamma(\frac{-2\lambda+n}{2})}{2^2\Gamma(-\lambda+n+1)\Gamma(\lambda+1)} C_{2(\lambda-n)}C_{-2\lambda}\delta_0(x).
    \end{align}
  \end{theorem}
  
  \begin{proof}
    By application of the Fourier transform, see Theorem \ref{FourierGeneralRiesz}, to \eqref{eq:ConvolutionIndentityGeneralRiesz} we obtain
    \begin{multline*}
      \mathcal{F}\big(  R^{2(\lambda-n)}_{A_{2(\lambda-n)},B_{2(\lambda-n)}}*
      R^{-2\lambda}_{A_{-2\lambda},B_{-2\lambda}}\big)(\xi)\\
        =2^{-2}\pi^n\frac{\Gamma(\frac{2\lambda-n}{2})\Gamma(\frac{-2\lambda+n}{2})}{\Gamma(-\lambda+n+1)\Gamma(\lambda+1)}
        r^{-4}(\xi)(C_{2(\lambda-n)}i_\xi\varepsilon_\xi
        +D_{2(\lambda-n)}\varepsilon_\xi i_\xi)(C_{-2\lambda}i_\xi\varepsilon_\xi
        +D_{-2\lambda}\varepsilon_\xi i_\xi).
    \end{multline*}
    Now the assumption $C_{2(\lambda-n)}C_{-2\lambda}=D_{2(\lambda-n)}D_{-2\lambda}$ 
    and $i_\xi\varepsilon_\xi+\varepsilon_\xi i_\xi=r^2(\xi)$, cf. Lemma \ref{Identities1}, 
    complete the proof.
  \end{proof}
  The assumption 
  $C_{2(\lambda-n)}C_{-2\lambda}=D_{2(\lambda-n)}D_{-2\lambda}$ is not empty. 
  For example take the values $(A_\lambda,B_\lambda)\st (1,1),(1,-1)$ 
  or $(\alpha_\lambda,\beta_\lambda)$, cf. \eqref{eq:BGCoefficients} . In each case the assumption 
  of Theorem \ref{ConvolutionIndentityGeneralRiesz} is satisfied. 

  A semi-group structure of $R^\lambda_{A_\lambda,B_\lambda}(x)$ in full generality is not possible, see Remark 
  \ref{Semi-GroupStructure} for an example and a general statement.

 \subsection{Integral operators}
  Integral transforms are widely used in mathematics. Most constructions of 
  integral operators based on integrating or taking 
  convolutions by kernel functions. 
  The Fourier transform is such an example, it arises by integration against $e^{i\langle x,\xi\rangle}$. 
  Also pseudo-differential operators arise as integral operators. 
  Special cases are differential operators, namely those having 
  polynomial kernels. Furthermore, some pseudo-differential operators have a certain 
  symmetry, for example the Knapp-Stein integral operators \cite{KnappStein} intertwine the corresponding 
  principal series. 

  Recall that $R^\lambda_{A_\lambda,B_\lambda}(x)$ is meromorphic in 
  $\lambda\in\C$ with only simple poles at $\lambda=-n-2\N_0$. 
  The integral operator 
  \begin{align}\label{eq:ConvolutionOperatorGeneralRiesz}
      \mathbb{T}_{\nu,(A_{\nu},B_{\nu})}:\Omega^p(\R^n)\to \Omega^p(\R^n)
  \end{align}
  is defined as convolution operator with $R^\lambda_{A_\lambda,B_\lambda}(x)$, i.e. 
  \begin{align*}
      (\mathbb{T}_{\nu,(A_{\nu},B_{\nu})}\omega)(x)
        \st \int_{\R^n} R^{2(\nu-n)}_{A_{2(\nu-n)},B_{2(\nu-n)}}(x-y)\omega(y)\dm y.
  \end{align*}
  These family $\mathbb{T}_{\nu,(A_{\nu},B_{\nu})}$ of pseudo-differential operators 
  is meromorphic in $\nu\in\C$ with only simple poles at $\nu=\frac n2-\N_0$. Now we describe their residues for 
  a class of integral operators.
  \begin{theorem}\label{ResidueIntegralOperatorGeneralRieszAssumption}
    Let $A_\lambda,B_\lambda$ holomorphic in $\lambda\in\C$ such that $C_{-n}=D_{-n}$. Then the residue 
    of $\mathbb{T}_{\nu,(A_{\nu},B_{\nu})}$ at $\nu=\frac n2-k$, for $k\in\N_0$, is 
    \begin{align}\label{eq:ResidueIntegralOperatorGeneralRieszAssumption}
      \Res_{\nu=\frac n2-k}(\mathbb{T}_{\nu,(A_{\nu},B_{\nu})}\omega)
        =\frac{(-1)^{k+1}\pi^{\frac n2}}{4^k k! \Gamma(\frac n2+k+1)}(C_{-n-2k}(\delta\dm)^k+D_{-n-2k}(\dm\delta)^k)\omega,
    \end{align}
    for $\omega\in\Omega^p(\R^n)$.
  \end{theorem}
  \begin{proof}
    This is a direct consequence of Corollary \ref{ResiduesGeneralRieszWithAssumption}. 
  \end{proof}

\section{Further results and applications}

 We focus on Riesz distributions $R^\lambda_{A_\lambda,B_\lambda}(x)$ for certain $A_\lambda,B_\lambda$, 
 and we show how to recover 
 the Knapp-Stein intertwining operators \cite{KnappStein} on functions and differential forms. Furthermore, 
 we discuss their relation to the well-known GJMS- and Branson-Gover 
 operators \cite{GJMS,BransonGover}, respectively.

 \subsection{Knapp-Stein intertwining operator on functions}  
  We show that the distribution 
  \begin{align}
    R^\lambda_{1,0}(x)=r^{\lambda-2}(x)i_x\varepsilon_x
  \end{align}
  when acting on $0$-forms 
  reduces to the Riesz distribution \eqref{eq:RieszDistribution}. 
  The Knapp-Stein intertwining operator on functions is 
  defined as convolution operator with the Riesz distribution.

  Let us briefly recall the Riesz distribution \cite{Riesz, GelfandShilov}, which is given for $\lambda\in\C$ with $\Re(\lambda)>-n$ by 
  \begin{align*}
      r^\lambda(x)=(x_1^2+\ldots+x_n^2)^{\frac{\lambda}{2}}.
  \end{align*}
  It extends to a tempered distribution meromorphic in $\lambda\in\C$ with 
  simple poles at $\lambda= -n-2k$ for $k\in\N_0$, and obeys the following properties:
  \begin{itemize}
      \item[(1)] Fourier transform: $\mathcal{F}(r^\lambda)(\xi)=2^{\lambda+n}\pi^{n/2}
         \frac{\Gamma(\frac{\lambda+n}{2})}{\Gamma(-\frac{\lambda}{2})}r^{-\lambda-n}(\xi)$.
      \item[(2)] Bernstein-Sato identity: $ \Delta(r^{\lambda+2}(x))=(\lambda+2)(\lambda+n)r^\lambda(x)$, 
         where $\Delta=\sum_i\partial_i^2$. 
      \item[(3)] Convolutionary inverse: 
         $(r^{2(\lambda-n)}*r^{-2\lambda})(x)=\pi^n\frac{\Gamma(\frac{2\lambda-n}{2})\Gamma(\frac{-2\lambda+n}{2})}{\Gamma(\lambda)\Gamma(n-\lambda)}\delta_0(x)$, where 
         $\delta_0(x)$ is the Dirac-distribution centered at the orign. 
      \item[(4)] Residues at $\lambda=-n-2\N_0$: 
         $\Res_{\lambda=-n-2k}(r^{\lambda}(x))=\frac{2\pi^{\frac n2}}{2^{2k}k!(\frac n2)_k\Gamma(\frac{n}{2})}\Delta^k \delta_0(x)$, 
         where $k\in\N_0$.
  \end{itemize}
  
  By Lemma \ref{Identities1} we obtain 
  \begin{align*}
      R^\lambda_{1,0}(x)=r^{\lambda-2}(x)i_x\varepsilon_x=r^\lambda(x).
  \end{align*}
  Hence, $R^\lambda_{A_\lambda,B_\lambda}(x)$ is a generalization of $r^\lambda(x)$. Note that one could 
  keep the parameters $(A_\lambda,B_\lambda)$ different from $(1,0)$ 
  which leads, as long as acting on functions, 
  to a different normalization of $r^{\lambda}(x)$ by the factor $A_\lambda$, since $B_\lambda$ 
  has no contribution on functions due to $i_x f=0$ for any function $f$.  
  
  It is an easy exercise, using the parameters $C_\lambda=\lambda$ and 
  $D_\lambda=-n$ (again the latter one has no contribution on functions), cf. \eqref{eq:Coefficients1}, that 
  Theorems \ref{FourierGeneralRiesz}, \ref{BernsteinSatoGeneralRiesz}, 
  \ref{ConvolutionIndentityGeneralRiesz} and \ref{ResiduesGeneralRiesz} specialize
  for $R^\lambda_{1,0}(x)$ to (1)-(4) given above. 
  
  \begin{remark}
    Note that we have only used the knowledge of the Fourier transform to achieve Theorem \ref{ResiduesGeneralRiesz}. 
    In \cite{GelfandShilov} a different 
    computation for $\Res_{\lambda=-n-2k}(r^\lambda(x))$ is 
    presented, which doesn't based on the Fourier transform. 
  \end{remark}
  
  The integral operator \eqref{eq:ConvolutionOperatorGeneralRiesz} reduces to 
  \begin{align}\label{eq:KnappSteinForFunctions}
    (\mathbb{T}_{\nu,(1,0)} f)(x)= \int_{\R^n}\abs{x-y}^{2(\nu-n)}f(y)dy, 
  \end{align}
  which is exactly the Knapp-Stein intertwining operator.  
  Its major property is that it intertwines the corresponding principal series \cite{KnappStein}. 
  Furthermore, from Theorem \ref{ResidueIntegralOperatorGeneralRieszAssumption}
  it follows that the residue of $\mathbb{T}_{\nu,(1,0)}$ at $\nu=\frac n2-k$, $k\in\N_0$, 
  is given by the differential operator 
  \begin{align*}
   \Res_{\nu=\frac n2-k}(\mathbb{T}_{\nu,(1,0)}) 
     = \frac{2\pi^{\frac n2}}{4^k k!\Gamma(\frac n2+k)}\Delta^k\delta_0.
  \end{align*}
  Note that our convention implies $\delta\dm=-\Delta$ when action on functions. Hence, residues of 
  $\mathbb{T}_{\nu,(1,0)}$ recover the well-known conformal powers of the Laplacian on function, 
  \cite{GJMS}.

 \subsection{Knapp-Stein intertwining operator for differential forms}\label{SubsectionKSDF}

 We next investigate the distribution 
 \begin{align}\label{eq:Riesz2}
   R^\lambda_{1,-1}(x)=r^{\lambda-2}(x)(i_x\varepsilon_x-\varepsilon_x i_x).
 \end{align}
 We show that the residues of integral operators $\mathbb{T}_{\nu,(1,-1)}$ are the Branson-Gover operators of 
 order $2N\in\N$, cf. \cite{BransonGover}, which are given by 
 \begin{align}\label{eq:BGOperator}
   L_{2N}^{(p)}=\alpha_N(\delta\dm)^N+\beta_N(\dm\delta)^N:\Omega^p(\R^n)\to\Omega^p(\R^n),
 \end{align}
 with coefficients 
 \begin{align}\label{eq:BGCoefficients}
   \alpha_\lambda\st\frac n2-p+\lambda,\quad \beta_\lambda\st\frac n2-p-\lambda.
 \end{align}
 Furthermore, we extend the intertwining property of $L_{2N}^{(p)}$, see for example \cite{FJS2}, 
 to $\mathbb{T}_{\nu,(1,-1)}$. We also could get the 
 intertwining property by a direct comparison of $\mathbb{T}_{\nu,(1,-1)}$ with the Knapp-Stein 
 intertwining operator studied by \cite{SpehVenkataramana}, see Remark 
 \ref{IntegralOperatorsComparision}. Finally we present an elementary 
 proof of positive-definiteness of scalar products induced by $\mathbb{T}_{\nu,(1,-1)}$, 
 cf. \cite{SpehVenkataramana}.
 
 Now, for $(A_\lambda,B_\lambda)=(1,-1)$ the coefficients \eqref{eq:Coefficients1} are  
 \begin{align*}
   C_\lambda=\lambda+2p=-2\alpha_{-\frac{\lambda+n}{2}},\quad 
     D_\lambda=-(\lambda+2n-2p)=-2\beta_{-\frac{\lambda+n}{2}}.
 \end{align*} 
 From Theorems \ref{FourierGeneralRiesz}, \ref{BernsteinSatoGeneralRiesz}, 
 \ref{ConvolutionIndentityGeneralRiesz} and Corollar \ref{ResiduesGeneralRieszWithAssumption}
 we obtain the following corollaries.
 \begin{corollary}\label{FourierRiesz2}
  The Fourier transform of $R^\lambda_{1,-1}(x)$ is given by
  \begin{align}\label{eq:FourierRiesz2}
   \mathcal{F}(R^\lambda_{1,-1})(\xi)
     =2^{\lambda+n}\pi^{\frac n2}\frac{\Gamma(\frac{\lambda+n}{2})}{\Gamma(-\frac{\lambda-2}{2})}
       R^{-\lambda-n}_{\alpha_{-\frac{\lambda+n}{2}},\beta_{-\frac{\lambda+n}{2}}}(\xi).
   \end{align}
 \end{corollary}
 
 \begin{corollary}\label{BernsteinSatoRiesz2}
  The distribution $R^\lambda_{1,-1}(x)$ satisfies the following Bernstein-Sato identiy:
  \begin{multline*}
   \Big[(\lambda+2n-2p)(\lambda+2p-2)\delta\dm 
     + (\lambda+2p)(\lambda+2n-2p-2)\dm\delta \Big]
     R^{\lambda}_{1,-1}(x)\\
   =-(\lambda-2)(\lambda+n-2)(\lambda+2p)(\lambda+2n-2p)R^{\lambda-2}_{1,-1}(x).
  \end{multline*}
 \end{corollary}
 Especially, the last corollary implies  
 \begin{multline}\label{eq:IteratedBernsteinSatoRiesz2}
   R^\lambda_{1,-1}(x)=\frac{(-1)^k}{4^k(\frac{\lambda}{2})_k(\frac{\lambda+n}{2})_k} 
     \Big[\frac{(\lambda+2p)}{(\lambda+2p+2k)}(\delta\dm)^k
     +\frac{(\lambda+2n-2p)}{(\lambda+2n-2p+2k)} (\dm\delta)^k  \Big]R^{\lambda+2k}_{1,-1}(x).
 \end{multline}

 \begin{corollary}\label{ResiduesRiesz2}
   The residue of $R^\lambda_{1,-1}(x)$ at $\lambda=-n-2k$ for $k\in\N_0$ is 
   \begin{align}\label{eq:ResiduesRiesz2}
     \Res_{\lambda=-n-2k}(R^\lambda_{1,-1}(x) )
      =\frac{(-1)^k 2\pi^{\frac n2}}{4^{k} k! \Gamma(\frac n2+k+1)}
       [\alpha_k (\delta\dm)^k+\beta_k (\dm\delta)^k]\delta_0(x),
   \end{align}
   which is up to a constant the Branson-Gover operator $L_{2k}^{(p)}$ acting on the Dirac-distribution. 
 \end{corollary}
 
 \begin{corollary}\label{ConvolutionIndentityRiesz2}
   The distribution $R^\lambda_{1,-1}(x)$ satisfies 
   \begin{align}\label{eq:ConvolutionIndentityRiesz2}
      R^{2(\lambda-n)}_{1,-1}*R^{-2\lambda}_{1,-1}(x)
      =\pi^n \frac{\Gamma(\frac{2\lambda-n}{2})\Gamma(\frac{-2\lambda+n}{2})}{\Gamma(-\lambda+n+1)\Gamma(\lambda+1)}\alpha_{\frac{-2\lambda+n}{2}}\alpha_{\frac{2\lambda-n}{2}}\delta_0(x)
   \end{align}
 \end{corollary}

 The integral operator \eqref{eq:ConvolutionOperatorGeneralRiesz} specializes to 
 \begin{align}\label{eq:ConvolutionOperatorRiesz2}
  (\mathbb{T}_{\nu,(1,-1)}\omega)(x)=\int_{\R^n}\abs{x-y}^{2(\nu-n-1)}
    (i_{x-y}\varepsilon_{x-y}-\varepsilon_{x-y}i_{x-y})\omega(y)\dm y
 \end{align}
 for $\omega\in\Omega^p(\R^n)$. From  Theorem \ref{ResidueIntegralOperatorGeneralRieszAssumption}
 we obtain the following.

 \begin{corollary}\label{ResiduesOfIntegralOperator}
  The residue of $\mathbb{T}_{\nu,(1,-1)}$ at $\nu=\frac n2-k$, $k\in\N_0$, is given by 
  \begin{align}\label{eq:ResiduesOfIntegralOperator}
    \Res_{\nu=\frac n2-k}\big( (\mathbb{T}_{\nu,(1,-1)}\omega)\big)
      =\frac{(-1)^k2 \pi^{\frac n2}}{4^k k! \Gamma(\frac n2+k+1)}L_{2k}^{(p)}\omega,
  \end{align} 
  where $\omega\in\Omega^p(\R^n)$.
 \end{corollary} 
 
 Now we show the intertwining property of $\mathbb{T}_{\nu,(1,-1)}$ with respect to the 
 principal series, acting on differential forms.
 For that let us introduce some notation.
 Let $G=SO_0(n+1,1,\R)$ the connected component of the group preserving the inner product 
 \begin{align*}
   2x_0 x_{n+1}+x_1^2+\cdots+x_n^2
 \end{align*}
 on $\R^{n+2}$. Let $P_+\st Stab(\R e_0)\subset G$ and 
 $P_-\st Stab(\R e_{n+1})\subset G$ be the stabilizer of the line $\R e_0$, respectively $\R e_{n=1}$, where 
 $\{e_0,\ldots,e_{n+1}\}$ is the standard basis of $\R^{n+2}$. The groups $P_\pm$ are parabolic in $G$ 
 and have Langlands decompositions $P_\pm=MAN_\pm$ with 
 $M\simeq SO(n,\R)$, $A\simeq \R^+$ and $N_\pm\simeq \R^n$. 
 Corresponding Lie algebras are denoted by 
 $\mathfrak{g}(\R)$ and 
 $\mathfrak{p}_\pm(\R)=\mathfrak{m}(\R)\oplus \mathfrak{a}(\R)\oplus\mathfrak{n}_\pm(\R)$ with 
 $\mathfrak{m}(\R)\simeq \mathfrak{so}(n,\R)$, $\mathfrak{a}(\R)\simeq \R$ and 
 $\mathfrak{n}_\pm(\R)\simeq \R^n$. 
 We define a representation of $P_+$ as an trivial extention to $N_+$ of the tensorproduct of the 
 standard representation $(\Lambda^p(\R^n)^*,\sigma_p)$ of $M$
 and a $1$-dimensional represenation $(\C_\lambda,\xi_\lambda)$ of $A$, 
 i.e. $P_+$ acts on 
 $V_{\lambda,p}\st \Lambda^p(\mathfrak{n}_-(\R))^*\otimes \C_{-\lambda}\simeq \Lambda^p(\R^n)^*\otimes \C_{-\lambda}$ by
 \begin{align*}
   \rho_{\lambda,p}:P_+=MAN_+&\to GL(V_{\lambda,p})\\
   man&\mapsto\{(v\otimes 1)\mapsto \sigma_p(m)v\otimes a^{-\lambda}\}.
 \end{align*}
 Now, the space of sections $\Gamma(G\times_{P_+,\rho_{\lambda,p}}V_{\lambda,p})$ is 
 equivalent to the space of $P_+$-equivariant functions 
 \begin{align*}
   \mathcal{C}^\infty(G,V_{\lambda,p})^{P_+}
    =\{f:G\to V_{\lambda,p}\mid f(gp^\prime)=\rho_{\lambda,p}((p^\prime)^{-1})f(g)\quad\forall p^\prime\in P_+\}.
 \end{align*}
 Let us denote 
 by $\pi_{\lambda,p}$ the regular left representation
(strictly speaking anti-representation)
 of $G$ on $\mathcal{C}^\infty(G,V_{\lambda,p})^{P_+}$, i.e. 
 $\pi_{\lambda,p}(g)f(g^\prime)\st f(g\cdot g^\prime)$. 
 \begin{lem}\cite[Lemma $2.3.3$]{FJS2}
  The infinitesimal action $\dm \pi_{\lambda,p}(E_j^+)$ on 
  $u\otimes \omega\in \mathcal{C}^\infty(\mathfrak{n}_-(\R))\otimes \Lambda^p(\mathfrak{n}_-(\R))^*\simeq\mathcal{C}^\infty(\R^n)\otimes \Lambda^p(\R^n)^*$ is given by 
  \begin{align*}
   \dm \pi_{\lambda,p}(E_j^+)(u\otimes\omega)(x)
    &=\Big(-\frac 12\sum_{k=1}^nx_k^2\partial_{x_j}+x_j(-\lambda
      +\sum_{k=1}^n x_k\partial_{x_k})\Big)(u)\otimes \omega\\
    &+\sum_{k=1}^nx_k u\otimes((E^+_j)^*\wedge  i_{E_k^-}-(E^+_k)^*\wedge i_{E_j^-})(\omega),
  \end{align*}
  where $\{E^\pm_j\}$ denotes a basis of $\mathfrak{n}_\pm(\R)\simeq \R^n$ and $\{(E^\pm_j)^*\}$ denotes its dual. 
 \end{lem}

 In these conventions the 
 Branson-Gover operators \eqref{eq:BGOperator} satisfy the infinitesimal intertwining property
 \begin{align}\label{eq:InvarianzBG}
  \dm \pi_{-\frac{n}{2}-N,p}(X) L_{2N}^{(p)}= L_{2N}^{(p)}\dm\pi_{-\frac{n}{2}+N,p}(X),
    \quad \forall\;X\in \mathfrak{so}(n+1,1,\R).
 \end{align}
 We show that the relation \eqref{eq:InvarianzBG} 
 extend to $\mathbb{T}_{\nu,(1,-1)}$.
 \begin{prop}\label{IntertwiningRiesz2}
  The intergral operators $\mathbb{T}_{\nu,(1,-1)}$ satisfy
  \begin{align}\label{eq:IntertwiningRiesz2}
    \dm \pi_{\nu-n,p}(X) \mathbb{T}_{\nu,(1,-1)}= \mathbb{T}_{\nu,(1,-1)}\dm\pi_{-\nu,p}(X),
    \quad \forall\;X\in \mathfrak{so}(n+1,1,\R).
  \end{align}
 \end{prop}
 \begin{proof}
  We prove the claim in the Fourier image. First note the identity and definition 
  \begin{align*}
    \mathcal{F}(  \dm \pi_{\lambda,p}(X)\omega)
      &=-i\Big(\frac 12 \xi_j\Delta_\xi -((\lambda+n)+\sum_{k=1}^n \xi_k\partial_{\xi_k})\partial_{\xi_j}\\
      &\quad\quad+\sum_{k=1}^n\partial_{\xi_k}((E^+_j)^*\wedge  i_{E_k^-}
       -(E^+_k)^*\wedge i_{E_j^-})  \Big)\mathcal{F}(\omega)\\
    &\st D_2(\lambda,j)\mathcal{F}(\omega).
  \end{align*}
  Hence $D_2(\lambda,j)$ is a second order differential operator on differential forms. 
  In order to show \eqref{eq:IntertwiningRiesz2} it remains to verify that the difference of 
  \begin{align*}
    \mathcal{F}(  \dm \pi_{\nu-n,p}(X) \mathbb{T}_{\nu,(1,-1)} )(\xi)
      &= D_2(\nu-n,j)\mathcal{F}(R^{2(\nu-n)}_{1,-1})(\xi)\mathcal{F}(\omega)(\xi)\\
    &=c_1D_2(\nu-n,j)\big(( r^{-2\nu+n-2}(\xi)( \alpha_{\frac{-2\nu+n}{2}}i_\xi\varepsilon_\xi
      +\beta_{\frac{-2\nu+n}{2}}\varepsilon_\xi i_\xi ))\mathcal{F}(\omega)(\xi) \big)
  \end{align*}
  and  
  \begin{align*}
    \mathcal{F}( \mathbb{T}_{\nu,(1,-1)}\dm\pi_{-\nu,p}(X))(\xi)
      &=\mathcal{F}(R^{2(\nu-n)}_{1,-1})(\xi)D_2(-\nu,j)\mathcal{F}(\omega)(\xi)\\
   &=c_1 r^{-2\nu+n-2}(\xi)( \alpha_{\frac{-2\nu+n}{2}}i_\xi\varepsilon_\xi
     +\beta_{\frac{-2\nu+n}{2}}\varepsilon_\xi i_\xi )D_2(-\nu,j)\mathcal{F}(\omega)(\xi).
  \end{align*}
  vanishes.
  Here $c_1$ is the constant arising from the Fourier transform 
  of $R^{2(\nu-n)}_{1,-1}(x)$. Performing the differentiation with $D_2(\nu-n,j)$ and $D_2(-\nu,j)$, 
  canceling the common contribution of $r^\mu$ for appropriate $\mu$ and 
  expanding everything in powers of $\nu$, the difference become 
  \begin{align*}
    P_3(\omega) \nu^3+P_2(\omega) \nu^2+P_1(\omega) \nu^1+P_0(\omega)=0,
  \end{align*}
  for some coefficients $P_i(\omega)$, $i=0,\ldots 3$, which are 
  differential operators acting on $\omega$. That it is of third degree in $\nu$ comes from the fact that 
  $D_2(\mu,j)$ is of second order containing a first order contribution with coefficient $\mu$ and 
  the coefficients $\alpha_\mu$ and $\beta_\mu$ are linear in $\mu$. Since the left-hand side 
  is a polynomial of order $3$, it remains to check 
  its vanishing at $4$ different point. Since Equation \eqref{eq:IntertwiningRiesz2} 
  is satisfied at all residues of $\mathbb{T}_{\nu,(1,-1)}$, that means at infinitely many, 
  due to Equation \eqref{eq:InvarianzBG}, the proof is complete.
 \end{proof}

 \begin{remark}\label{IntegralOperatorsComparision}
  The intertwining property of the intergral operator $\mathbb{T}_{\nu,(1,-1)}$ was already studied in 
  \cite[Section $4$]{SpehVenkataramana}. Actually they have studied the 
  Knapp-Stein intertwining operator for differential forms. It just remains to show that our formula for 
  $\mathbb{T}_{\nu,(1,-1)}$ matches with that in \cite{SpehVenkataramana}. 
  This can be seen by noting that, when acting on a vector $Y$ in $\R^n$ 
  (a dual $1$-form), our algebraic action
  \begin{align*}
    \frac{i_x\varepsilon_x-\varepsilon_x i_x}{\abs{x}^2}Y=Y-2\frac{\langle Y,x\rangle}{\abs{x}^2}x
  \end{align*}
  becomes a reflection of $Y$ in the hyperplane (through the origin) orthogonal to $x$, compare with 
  \cite[Lemma $2.1$]{SpehVenkataramana}. Note that this formula extends
naturally to $p$-forms. Note also that 
  our notation for the algebraic action enables us to compute explicitly the residues
  of $\mathbb{T}_{\nu,(1,-1)}$, or even in a more general setting for $\mathbb{T}_{\nu,(A_\lambda,B_\lambda)}$.
 \end{remark}
 
\begin{remark}\label{ComplementarySeries}
 The intertwining operator $\mathbb{T}_{\nu,(1,-1)}$ defines an invariant pairing between
 two induced representations in natural duality (via the natural $L^2$ pairing), namely
 $\pi_{\nu - n, p}$ and $\pi_{-\nu,p}$ - here we consider real values of $\nu$, and
 from Corollary \ref{FourierRiesz2} we see that the interval where this pairing is
positive-definite (i.e. the interval for the unitary complementary series) is
exactly $|\lambda| < (n/2) - p$, with $\lambda = \nu - (n/2)$. 
Indeed, the invariant Hermitian form is on the Fourier transform side given as the
natural $L^2$ expression (the Fourier transform of $(\mathbb{T}_{\nu,(1,-1)}\omega, \omega)$)
\begin{align*}
2^{2\lambda}\pi^{n/2}\frac{\Gamma(\lambda)}{\Gamma((n/2) +1-\lambda)}
\int_{\mathbb{R}^n} |\xi|^{-2\lambda - 2}(((n/2) - p - \lambda) i_{\xi}\varepsilon_{\xi} + ((n/2) - p + \lambda)
\varepsilon_{\xi} i_{\xi})\widehat{\omega}, \widehat{\omega}) d\xi
\end{align*}
which we see as positive definite in the interval indicated; furthermore, for
$$0 < \lambda < (n/2) - p$$ the density here is locally integrable. Note that
in this case $\mathbb{T}_{\nu,(1,-1)}$ gives an equivalence between two unitary
representations in the complementary series.

This is consistent
with the formulas in \cite{BOO}, and gives a new and elementary proof of the size of
the unitary complementary series corresponding to $p$-forms. Indeed, in \cite{BOO}
the eigenvalues of the intertwining operator in its compact picture on the sphere
are shown to be labeled by integers $j \geq 1 $ and $q = 0,1$ (corresponding
to exact and co-exact differential forms respectively) as
$$Z(j,q,\lambda) = \frac{\Gamma((n/2) + j + \lambda)}{\Gamma((n/2) + j - \lambda)}
 \frac{\Gamma((n/2) - p + q + \lambda)}{\Gamma((n/2) - p + q - \lambda)}$$
suitably normalized. The pair $(j,q)$ corresponds to the highest weight $$(j,1,1, ,\dots, 1, q, 0, ,\dots)$$
of an irreducible representation of $K = SO(n+1)$, and $\lambda$ is the same parameter as before.
Note that the ratio between this eigenvalue for $q = 0$ and $q = 1$ is exactly
$((n/2) - p - \lambda)/((n/2) - p + \lambda)$. 
\end{remark}

 Let us close this subsection with a comment about the impact of the Bernstein-Sato 
 identity, see Corollary \ref{BernsteinSatoRiesz2}, for $R^\lambda_{1,-1}(x)$ to 
 a recurrence relation among Branson-Gover operators. This recurrence relation can 
 also be directly obtained from \eqref{eq:BGOperator}. However, we want to demonstrate its appearence 
 from the Bernstein-Sato identity. 
 \begin{prop} 
   The Branson-Gover operators $L_{2N}^{(p)}$, for $N\in\N$, on $\R^n$ satisfy
   \begin{align}\label{eq:RecurrenceOfBG}
     L_{2N}^{(p)}=\Big(\frac{\alpha_{N}}{\alpha_{N-1}}\delta\dm +\frac{\beta_N}{\beta_{N-1}}\dm\delta \Big)
       L_{2N-2}^{(p)}.
   \end{align}
   By convention we set $L_0^{(p)}\st \alpha_0\id$. 
 \end{prop}
 \begin{proof}
  First observe that Corollary \ref{BernsteinSatoRiesz2} implies
  \begin{align*}
    \mathbb{T}_{\nu,(1,-1)}\omega(x)&=\int_{\R^n} R^{2(\nu-n)}_{1,-1}(x-y)\omega(y)\dm y\\
    &=\frac{1}{(2\nu-2n)(2\nu-n)}\Big(\frac{2\nu-2n+2p}{2\nu-2n+2p+2}\delta\dm
      +\frac{2\nu-2p}{2\nu-2p+2}\dm\delta  \Big)\\
    &\times\int_{\R^n}R^{2(\nu-n+1)}_{1,-1}(x-y)\omega(y)\dm y
  \end{align*}
  Now we take the residue at $\nu=\frac n2-N$ for $N\in\N_0$ using 
  Equation \eqref{eq:ResiduesOfIntegralOperator}.
  The residue of the left-hand side is
  \begin{align*}
   \frac{(-1)^N2\pi^{\frac n2}}{4^N N!\Gamma(\frac n2+N+1)}L_{2N}^{(p)},
  \end{align*} 
  while the right-hand side has the residue 
  \begin{align*}
    \frac{(-1)^{N}2\pi^{\frac n2}}{4^{N}N!\Gamma(\frac n2+N+1)}
      \Big(\frac{n-2p+2N}{n-2p+2N-2}\delta\dm+\frac{n-2p-2N}{n-2p-2N+2}\dm\delta\Big)
       L_{2N-2}^{(p)}.
  \end{align*}
  Now a cancelation of common factors completes the proof.
 \end{proof}

\section{Some remarks and further applications}

 This sections collect some observations concerning $R^\lambda_{A_\lambda,B_\lambda}(x)$.

 \begin{remark}\label{Semi-GroupStructure}
   The Riesz distribution $R^\lambda_{1,1}(x)$, when appropriately normalized, satisfies a semi-group 
   property, i.e. for 
   \begin{align*}
     \bar{R}^\lambda_{1,1}(x)\st\frac{\Gamma(-\frac{\lambda-n}{2})}{2^\lambda\pi^{\frac n2}\Gamma(\frac{\lambda}{2})} R^{\lambda-n}_{1,1}(x)
   \end{align*}
   we have, whenever it makes sense, 
   \begin{align*}
     \bar{R}^\lambda_{1,1}*\bar{R}^\nu_{1,1}(x)=\bar{R}^{\lambda+\nu}_{1,1}(x).
   \end{align*}
   In order to prove a semi-group structure for the family of 
   Riesz distributions $R^\lambda_{A_\lambda,B_\lambda}(x)$ one has to allow 
   $A_\lambda,B_\lambda$ to be meromorphic in $\lambda\in\C$. 
   More precisely, define 
   \begin{align*}
     \bar{R}^\lambda_{A_\lambda,B_\lambda}(x)\st-\frac{\Gamma(-\frac{\lambda-n-2}{2})}{2^{\lambda-1}\pi^{\frac n2}\Gamma(\frac{\lambda}{2})} R^{\lambda-n}_{A_{\lambda-n},B_{\lambda-n}}(x),
   \end{align*}
   and show in the Fourier image that 
   \begin{align*}
     \mathcal{F}(\bar{R}^\lambda_{A_\lambda,B_\lambda}*\bar{R}^\nu_{A^\prime_\nu,B^\prime_\nu})(\xi)
       &=\mathcal{F}(\bar{R}^\lambda_{A_\lambda,B_\lambda})(\xi)
         \mathcal{F}(\bar{R}^\nu_{A^\prime_\nu,B^\prime_\nu})(\xi)\\
     &=r^{-(\lambda+\nu)-2}(\xi)(C_{\lambda-n}C^\prime_{\nu-n}i_\xi\varepsilon_\xi
       +D_{\lambda-n}D^\prime_{\nu-n}\varepsilon_\xi i_\xi)
   \end{align*}
   and 
   \begin{align*}
     \mathcal{F}(\bar{R}^{\lambda+\nu}_{A^{\prime\prime}_{\lambda+\nu},B^{\prime\prime}_{\lambda+\nu}})(\xi)
       &=r^{-(\lambda+\nu)-2}(\xi)(C^{\prime\prime}_{\lambda+\nu-n}i_\xi\varepsilon_\xi
       +D^{\prime\prime}_{\lambda+\nu-n}\varepsilon_\xi i_\xi)
   \end{align*}
   do agree iff $C^{\prime\prime}_{\lambda+\nu-n}=C_{\lambda-n}C^\prime_{\nu-n}$ and 
   $D^{\prime\prime}_{\lambda+\nu-n}=D_{\lambda-n}D^\prime_{\nu-n}$ hold. Note that a tuple $(A_\lambda,B_\lambda)$ 
   will give a tuple $(C_\lambda,D_\lambda)$, see \eqref{eq:Coefficients1}. Eqivalently, 
   using \eqref{eq:Coefficients1} we have to solve 
   \begin{align*}
     C_{\lambda-n}C^\prime_{\nu-n}=(\lambda+\nu-n+p)A^{\prime\prime}_{\lambda+\nu-n}
       -p B^{\prime\prime}_{\lambda+\nu-n},\\
     D_{\lambda-n}D^\prime_{\nu-n}=-(n-p)A^{\prime\prime}_{\lambda+\nu-n}
       +(\lambda+\nu-p)B^{\prime\prime}_{\lambda+\nu-n}
   \end{align*}
   for $A^{\prime\prime}_{\lambda+\nu-n},B^{\prime\prime}_{\lambda+\nu-n}$. The unique solution is given by 
   \begin{align*}
     A^{\prime\prime}_{\lambda+\nu-n}&\st \frac{(\lambda+\nu-p)C_{\lambda-n}C^\prime_{\nu-n}+p D_{\lambda-n}D^\prime_{\nu-n}}{(\lambda+\nu)(\lambda+\nu-n)}\\
     B^{\prime\prime}_{\lambda+\nu-n}&\st \frac{(n-p)C_{\lambda-n}C^\prime_{\nu-n}+(\lambda+\nu-n+p) D_{\lambda-n}D^\prime_{\nu-n}}{(\lambda+\nu)(\lambda+\nu-n)}.
   \end{align*}
   Hence we obtain the semi-group property, whenever it makes sense,  
   \begin{align*}
     \bar{R}^\lambda_{A_\lambda,B_\lambda}*\bar{R}^\nu_{A^\prime_\nu,B^\prime_\nu}
      =\bar{R}^{\lambda+\nu}_{A^{\prime\prime}_{\lambda+\nu},B^{\prime\prime}_{\lambda+\nu}}.
   \end{align*}
 \end{remark}

 \begin{remark}
  We have seen in Theorem \ref{FourierGeneralRiesz} that the Fourier transform preserves the family of Riesz 
  distribution $R^\lambda_{\cdot,\cdot}(x)$. In general it holds that the pairs 
  $(A_\lambda,B_\lambda)$ and $(C_{\lambda},D_{\lambda})$, see \eqref{eq:Coefficients1}, have nothing in common. 
  However, for $A_\lambda\st \alpha_\lambda$ and 
  $B_\lambda\st \beta_\lambda$ we have $C_\lambda=-\lambda\alpha_{-\lambda-n}$ and 
  $D_\lambda=-\lambda\beta_{-\lambda-n}$, hence by Theorem \ref{FourierGeneralRiesz} we have 
  \begin{align*}
    \mathcal{F}(R^\lambda_{\alpha_\lambda,\beta_\lambda})(\xi)
     =c_\lambda R^{-\lambda-n}_{\alpha_{-\lambda-n},\beta_{-\lambda-n}}(\xi)
  \end{align*}
  for some constant $c_\lambda$. 
 \end{remark}

 \begin{remark}
  Similar operators to the
ones we study here have been considered in connection with
Ahlfors operators and generalized Beurling-Ahlfors operators,
see \cite{IM}; these arise in the study of quasi-regular mappings generalizing
quasi-conformal mappings in two dimensions. Of interest here are
new $L^p$ - estimates for the Beurling-Ahlfors operator on differential forms.
 \end{remark}

 \begin{remark}{(\bf Structure of GJMS operators)}
   Let $(M,g)$ be a manifold. The GJMS operators, \cite{GJMS}, 
   \begin{align*}
     P_{2N}:\mathcal{C}^\infty(M)\to\mathcal{C}^\infty(M),
   \end{align*} 
   arises as residues of the scattering operator 
   $S(s):\mathcal{C}^\infty(M)\to\mathcal{C}^\infty(M)$, \cite{GZ}. 
   The latter one is a generalization of the intertwining 
   integral operators $\mathbb{T}_{\nu,(1,0)}:\mathcal{C}^\infty(\R^n)\to\mathcal{C}^\infty(\R^n)$ 
   considered on $0$-forms, see \eqref{eq:KnappSteinForFunctions}. Since each $P_{2N}$ is polynomial in 
   second order differential operators $M_{2k}$ \cite{J}, $k=0,\ldots,2N$, does this structure extend to $S(s)$? 
   Explicit formulae for those second order differential operators exist on flat and on Einstein manifolds. 
 \end{remark}
 
 \begin{remark}{(\bf Structure of Branson-Gover operators and conformal powers of the Dirac operator)}
   The Branson-Gover operators \cite{BransonGover}, again conjectural, and 
   conformal powers of the Dirac operator \cite{GMP1} are also residues of certain scattering operators. 
   Less is known about their structure. One might expect that some aspects of our 
   results here in the model case carry over to the case of Riemannian manifolds and conformal 
   geometry here; in particular the nature of a scattering operator should resemble our convolutions with Riesz kernels.  
 \end{remark}
 
 \begin{remark}
It is clear that given the large amount of analysis based on the Riesz distributions on functions,
one may expect similar applications (e.g. fundamental solutions, wave equations, heat equations,
$L^p$-mapping properties)
 of our formulas here; we plan to consider some of these
applications in another paper.
  Along similar lines: can one study heat equations for the integral kernels $\mathbb{T}_{\nu,(A_\nu,B_\nu)}$? 
 \end{remark}


\begin{thebibliography}{GJMS92}

\bibitem[AG11]{AG}
E.~Aubry and C.~Guillarmou, \emph{Conformal harmonic forms, {B}ranson-{G}over
  operators and {D}irichlet problem at infinity}, Journal of the European
  Mathematical Society (2011), 911--957.

\bibitem[Ber71]{Bernstein}
I.N.~Bernshtein, \emph{Modules over a ring of differential operators. {S}tudy
  of the fundamental solutions of equations with constant coefficients},
  Functional analysis and its applications \textbf{5} (1971), no.~2, 89--101.

\bibitem[BG05]{BransonGover}
T.~Branson and A.R.~Gover, \emph{Conformally invariant operators, differential
  forms, cohomology and a generalisation of {$Q$-}curvature}, Communications in
  Partial Differential Equations \textbf{30} (2005), no.~11, 1611--1669.

\bibitem[B\'O{}\O{}]{BOO}
T.~Branson, G.~\'O{}lafsson and B.~\O{}rsted, \emph{Spectrum Generating Operators and
Intertwining Operators for Representations Induced from a Maximal Parabolic Subgroup},
Journal of Functional Analysis \textbf{135} (1996), no.~1, 163--205.

\bibitem[C{\O}14]{CO}
J-L.~Clerc and B.~{\O}rsted, \emph{Conformal covariance for the powers of the {D}irac operator},
  ArXiv e-prints (2014),  1--20, \url{http://arxiv.org/abs/1409.4983}.

\bibitem[FJS16]{FJS2}
M.~Fischmann, A.~Juhl, and P.~Somberg, \emph{{Conformal symmetry breaking
  differential operators on differential forms}}, ArXiv e-prints (2016), 1--99,
  \url{http://arxiv.org/abs/1605.04517}.

\bibitem[GJMS92]{GJMS}
C.R.~Graham, R.W.~Jenne, L.~Mason, and G.~Sparling, \emph{Conformally invariant
  powers of the {L}aplacian, {I}: {E}xistence}, Journal of the London
  Mathematical Society \textbf{2} (1992), no.~3, 557--565.

\bibitem[GMP12]{GMP1}
C.~Guillarmou, S.~Moroianu, and J.~Park, \emph{Bergman and {C}alder\'on
  projectors for {D}irac operators}, Journal of Geometric Analysis (2012),
  1--39, \url{http://dx.doi.org/10.1007/s12220-012-9338-9}.

\bibitem[GS64]{GelfandShilov}
I.M.~Gelfand and G.E.~Shilov, \emph{Generalized {F}unctions, {V}olume {I}:
  {P}roperties and {O}perations}, Academic Press Inc., 1964.

\bibitem[GZ03]{GZ}
C.R.~Graham and M.~Zworski, \emph{Scattering matrix in conformal geometry},
  Inventiones mathematicae \textbf{152} (2003), 89--118.

\bibitem[IM93]{IM} T.~Iwaniec and G.~Martin, \emph{Quasi-regular mappings in
even dimensions}, Acta Math. \textbf{170} (1993), 29--81.

\bibitem[J10]{J} A.~Juhl, \emph{On conformally covariant powers of the {L}aplacian}, ArXiv e-prints (2010), 1--61,
  \url{https://arxiv.org/pdf/0905.3992v3.pdf}.


\bibitem[KS71]{KnappStein}
A.W.~Knapp and E.M.~Stein, \emph{{Intertwining {O}perators for {S}emisimple
  {G}roups}}, Annals of Mathematics \textbf{93} (1971), no.~3, 489--578.

\bibitem[KV91]{KV}
J.A.C.~Kolk and V.S.~Varadarajan, \emph{Riesz distributions},
  Mathematica Scandinavica \textbf{68} (1991), 273--291.

\bibitem[Rie49]{Riesz}
M.~Riesz, \emph{{L'int\'egral de {R}iemann-{L}iouville et le probl\`em de
  {C}auchy}}, Acta Mathematica \textbf{81} (1949), no.~1, 1--222.

\bibitem[SS72]{SatoShintani}
M.~Sato and T.~Shintani, \emph{On {Z}eta functions associated with
  prehomogeneous vector spaces}, Proceedings of the National Academy of
  Sciences \textbf{69} (1972), no.~5, 1081--1082.

\bibitem[SV11]{SpehVenkataramana}
B.~Speh and T.N.~Venkataramana, \emph{{Discrete components of some
  complementary series}}, Forum Mathematicum \textbf{23} (2011), 1159--1187.

\end{thebibliography}
\end{document}